\documentclass[11pt]{amsart}
\usepackage{amsmath, amsthm, amssymb}
\usepackage{graphicx}
\usepackage{mathbbol}
\usepackage{txfonts}
\usepackage[usenames]{color}
\usepackage{srcltx} % JD - to jump between source and preview

\newtheorem{thm}{Theorem}[section]
\newcommand{\bt}{\begin{thm}}
\newcommand{\et}{\end{thm}}

\newtheorem{cor}[thm]{Corollary}   %remember switch all {coro} to {cor}
\newcommand{\bc}{\begin{cor}}
\newcommand{\ec}{\end{cor}}

\newtheorem{lem}[thm]{Lemma}   %remember to switch all {lem} to {lem}
\newcommand{\bl}{\begin{lem}}
\newcommand{\el}{\end{lem}}

\newtheorem{prop}[thm]{Proposition}
\newcommand{\bp}{\begin{prop}}
\newcommand{\ep}{\end{prop}}

\newtheorem{defn}[thm]{Definition}
\newcommand{\bd}{\begin{defn}}    % This produces an error????    
\newcommand{\ed}{\end{defn}}

\newtheorem{rmrk}[thm]{Remark}   %remember to switch all {rmrk} to {rmrk}

\newcommand{\br}{\begin{rmrk}}
\newcommand{\er}{\end{rmrk}}

\newtheorem{example}[thm]{Example}

\newcommand{\mGHto}{\stackrel { \textrm{mGH}}{\longrightarrow} }
\newcommand{\GHto}{\stackrel { \textrm{GH}}{\longrightarrow} }
\newcommand{\Fto}{\stackrel {\mathcal{F}}{\longrightarrow} }
\newcommand{\be}{\begin{equation}}

\newcommand{\ee}{\end{equation}}

\newcommand{\N}{\mathbb{N}}

\newcommand{\R}{\mathbb{R}}

\newcommand{\E}{\mathbb{E}}

\newcommand{\dist}{\operatorname{dist}}
\newcommand{\diam}{\operatorname{Diam}}

\newcommand{\Fm}{{\mathcal F}}
\newcommand{\Hm}{{\mathcal H}}

\newcommand{\Scal}{{\rm Scal}} % new Jan 2016
\newcommand{\disjointunion}{\sqcup}
\newcommand{\inv}{^{-1}}
\newcommand{\Lip}{\operatorname{Lip}}

\newcommand{\mass}{{\mathbf M}}

\newcommand{\vol}{\operatorname{Vol}}

\def\rmin{r_{\textrm{min}}}

\DeclareMathOperator{\RS}{RotSym}

\begin{document}

\title{Sequences of three dimensional manifolds with positive scalar curvature}

\author{J. Basilio}
\thanks{J. Basilio was partially supported as a doctoral student by NSF DMS 1006059.}
\address{Pasasdena City College}
\email{jorge.math.basilio@gmail.com}

\author{C. Sormani}
\thanks{C. Sormani was partially supported by NSF DMS 1612049.}
\address{City University of New York Graduate Center and Lehman College}
\email{sormanic@gmail.com}

%\date{February 2017}


%49Q15 (Geometric measure and integration theory, integral and normal currents)

%\subjclass[2000]{49Q15}

\begin{abstract} 
We develop two new methods of constructing sequences of manifolds
with positive scalar curvature that converge in the Gromov-Hausdorff and Intrinsic 
Flat sense to limit spaces with "pulled regions".   The examples created rigorously using these
methods were announced
a few years ago and have influenced the statements of some of Gromov's conjectures 
concerning sequences of manifolds with positive scalar curvature.
Both methods extend 
the notion of ``sewing along a curve'' developed in prior work of the authors with
Dodziuk to create limits that are pulled string spaces.  The first method allows us to sew 
any compact set in a fixed initial manifold to
create a limit space in which that compact set has been scrunched to a single point.  The
second method allows us to edit a sequence of regions or curves in a sequence of distinct manifolds. 
  \end{abstract}

\maketitle

\section{Introduction}

In \cite{Gromov-Plateau}, Gromov challenged mathematicians to explore generalized notions of scalar curvature that persist under Gromov-Hausdorff and Intrinsic Flat convergence.   The most simply stated geometric definition of scalar curvature at a point,
\be\label{E:scalarlimit} % SCALAR CURVATURE AS VOLUME
\Scal(p) = \lim_{r \to 0^{+}} 30\left(\frac{\vol_{\E^{3}}(B(0,r)) - \Hm^{3}(B(p,r))}{r^{2}\cdot \vol_{\E^{3}}(B(0,r))} \right),
\ee
uses Hausdorff measure to replace volume does not behave well under convergence.  In joint work of the authors with Dodziuk, we constructed a sequence of manifolds with positive scalar curvature which converged in the Gromov-Hausdorff and Intrinsic Flat sense to a limit space for which this limit is negative at a point \cite{BDS}.   That example was constructed using a method we called ``sewing along a curve'' and the limit space was a standard three dimensional sphere in which one of the closed geodesics was "pulled to a point".  In that paper we announced additional examples which we now present here, in which we "sew" arbitrary compact sets and create limit spaces
where the entire compact set is "pulled to a point".  The examples created here give new insight into the variety 
of spaces that can appear as limits of manifolds with positive scalar curvature.  Indeed the existence of these examples
and additional examples by the authors which will appear in upcoming work \cite{BS-Tori}, has lead to refinement of Gromov's conjectures in \cite{Sormani-Scalar} and new proposed conjectures at the {\em IAS Emerging Topics on Scalar Curvature and Convergence} organized by Gromov and the second author in Fall 2018.   

The most important theorem concerning manifolds with positive scalar curvature is the Schoen-Yau Positive Mass Theorem \cite{Schoen-Yau-positive-mass}.  This theorem states that a complete noncompact manifold with positive scalar curvature that is asymptotically flat must have positive ADM mass.  The ADM mass is the limit
of Hawking masses of increasingly large round spheres in the asymptotically flat region:
\be
 \mathrm{m}_{\mathrm{ADM}}(M) = \lim_{s \to \infty} \mathrm{m}_{\mathrm{H}}(\Sigma_s) \in [0,\infty],  
\ee
where the Hawking mass of a surface is defined using the integral of the mean curvature
of the surface squared as follows:
\be
 \mathrm{m}_{\mathrm{H}}(\Sigma) = \frac{1}{2}\left(\frac{A}{\omega_{2}}\right)\left(1-\frac{1}{4\pi}\int_\Sigma \left(\frac{H}{2}\right)^{2}\right),
 \ee
Schoen-Yau also prove the rigidity statement that if the ADM mass is zero then the manifold is isometric to
Euclidean space.  In recent years there has been much work exploring how this theorem is stable under 
various notions of convergence by the second author and Bamler, Huang,  Jauregui, Lee, LeFloch, Mantoulidis, Schoen, Sakovich, Stavrov, and others \cite{Bamler-16} \cite{LeeSormani1}\cite{LeFloch-Sormani-1} \cite{HLS}\cite{Jauregui-Lee}\cite{Sormani-Stavrov-1}\cite{Mantoulidis-Schoen}.  \footnote{We welcome additional suggested citations.}

In this paper we construct sequences of asymptotically flat manifolds with positive scalar curvature, which converge smoothly outside of a compact set to Euclidean space, but which have various sets within them sewn to points so that
the limit space does not satisfy the Schoen-Yau rigidity statement.   To construct these sequences we need to develop a second method of sewing manifolds, this time we don't start with a fixed manifold and creating sewing it more and more tightly, but instead start with a sequence of manifolds and sew each term in the sequence more and more tightly.   This is referred to as Method II within.

The paper begins with a review of the notion of a pulled string space first introduced by Burago in discussions with the second author while they were working on ideas leading to  \cite{Burago-Ivanov-area} with Ivanov.  Intuitively these spaces are like pieces of cloth in which one string has been pulled tightly, so that it is identified to a point.   Here we introduce the idea of a pulled metric space, in which an entire compact set has been pulled to a single point and a method we call scrunching which can be applied to prove a sequence of manifolds converges to a given pulled metric space.  This first section is pure metric geometry and does not involve any scalar curvature.  It is somewhat technical if one does not already know the methods Gromov-Hausdorff and Intrinsic Flat convergence.  A review of the necessary background can be found in \cite{BDS} so we do not repeat it here.

In Section 2, we introduce our Method I for creating sequences of manifolds with positive scalar curvature that converge to pulled metric spaces.   We begin by reviewing the construction of tunnels of positive scalar
curvature found by Schoen-Yau \cite{Schoen-Yau-tunnels} and Gromov-Lawson \cite{Gromov-Lawson-tunnels} (cf the appendix to \cite{BDS}).  We review also the method of sewing along a curve by placing the tunnels in a paired pattern along a fixed curve in a fixed manifold to create a new manifold with positive scalar curvature.  One can then sew along the curve more and more tightly, by taking the tunnels smaller and closer together in a precise way, to create a sequence of manifolds with positive scalar curvature that converges to a limit space where that curve has been pulled to a point.  All this was done by the authors with Dodziuk in \cite{BDS}.   In our new Method I we extend this to arbitrary compact sets rather than just curves in a fixed Riemannian manifold.  This involves the development of a new pattern for placing the tunnels, which is perhaps somewhat similar to a pattern the authors used with Kazaras in \cite{BKS} except that we are sewing the compact regions tightly to points in this method.  In Proposition~\ref{sewn-region} we prove that we obtain a manifold with positive scalar curvature that is sewn.   We prove Method I works to produce a pulled limit space in Theorem~\ref{thm-seq-sewn}.   

In Section 3  we apply Method I to present two examples Examples~\ref{sphere-geod} and~\ref{sphere-equator}, although one can easily imagine how it can be applied in many other ways.  The limit in Example~\ref{sphere-geod} is a standard three dimensional sphere with a single geodesic pulled to a point.  The limit in Example~\ref{sphere-equator} is a a standard three dimensional sphere with the equatorial sphere pulled to a point.  One might in fact create sequences which pull any compact set in a standard three sphere to a point, or indeed any
compact set with positive curvature within an arbitrary manifold.  It is crucial that the compact set being sewn to a point
has small balls isometric to balls in spheres of constant sectional curvature but that constant may be arbitrarily small as long as it is positive (see Proposition~\ref{sewn-region} for the precise requirements).   

In Section 4 we develop Method II which provides a method of sewing a sequence of compact sets in a sequence of distinct manifolds.  See Theorem~\ref{prop-seq-sewn} for the precise statement.   Note that this theorem is proven quite generally and does not require positive scalar curvature.  It is about when sequences of manifolds created using a scrunching or sewing of regions converges to a certain space.  When combined with Proposition~\ref{sewn-region} it can be applied to produce new sequences of manifolds with positive scalar curvature.   This method is needed to construct the examples related to the positive mass theorem, because one cannot sew Euclidean space.  One can only sew regions with strictly positive sectional curvature.  Method II allows us to take sequences of manifolds with positive scalar curvature each with
a compact region of positive sectional curvature to produce a Euclidean limit space that has a compact sewn to a point.

In Section 5, we apply Method II to present Examples 6.7-6.9 in which a sequence of asymptotically
flat manifolds with positive scalar curvature and ADM mass converging to 0 converge in the pointed Gromov-Hausdorff and intrinsic flat sense to Euclidean Space with a compact set pulled to a point.  The construction begins using a sequence of
manifolds found in work of the second author with Lee in \cite{LeeSormani2} of smooth spherically symmetric manifolds with positive ADM mass converging to 0 that have rings of constant positive sectional curvature.  These rings are the compact sets that are sewn so that in the limit the ring is pulled to a point.

Some of this research was completed at the CUNY Graduate Center as part of the the first author's doctoral dissertation  completed under the supervision of Dodziuk and the second author.  A few of the examples were announced there, and also in the second author's survey \cite{Sormani-Scalar}, and have been presented many times.  This is the first time the work has been completed rigorously for publication.  It should be noted that additional examples constructed using Method II and announced in the first author's thesis and \cite{Sormani-scalar} concerning limits of almost nonegative scalar curvature
will appear rigorously in upcoming work by the authors \cite{BS-Tori}. We would like to thank Jeff Jauregui, Marcus Khuri, Sajjad Lakzian, Dan Lee, Raquel Perales, Conrad Plaut, Catherine Searle, Dan King, and Philip Ording 
for their interest in this work.

\section{Converging to Pulled Metric Spaces}

In this paper the limits of our sequences of Riemannian manifolds will
no longer be Riemannian manifolds.  They will be pulled metric spaces
created by taking Riemannian manifold and ``pulling a compact set to a 
point''.  We review this notion in the first subsection and then provide
a subsection describing a setting when a sequence of 
Riemannian manifolds converges to such a pulled metric space.
Within this second subsection we recall key methods
used to prove Gromov-Hausdorff, metric measure, and intrinsic flat 
convergence as needed.   We also recall many lemmas proven in the
author's joint work with Dodziuk \cite{BDS}.  Doctoral students are recommended
to read that paper before this one for a thorough review of all the background material.

Note that this section does not involve scalar curvature in any way.   It develops the metric geometry
required to prove our new examples of sequences of manifolds with scalar curvature bounds
converge as we claim they converge.  These techniques will be applied elsewhere in the future
 in upcoming work of the authors.

\subsection{Pulled Metric Spaces} \label{sect-pulled-string}
%%%%%%%%%%%%%%%%%%%%%%%%%%%%%%%%%%
% Pulled Metric Spaces

A special kind of pulled metric space called a pulled string space
was first described to the second author by
Dimitri Burago when they were working together on
ideas that lead towards an intriguing paper of Burago and Ivanov \cite{Burago-Ivanov-Area}.   One
starts with a standard square patch of cloth, $X=[0,1]^2$, and a 
string $C:[0,1]\to X$ where $C(t)=(t, 1/2)$.  One creates the
pulled string space
\be
Y := (X \setminus K) \disjointunion \{p_0\}, \qquad p_0 \in K \, \textrm{fixed}, 
\ee
where $K$ is the image of $C$.  This pulled string space, $Y$,
may intuitively be viewed as the square patch of cloth
with a single thread (identified by the curve $C$) which 
has been pulled tight.   Such pulled string spaces starting 
from an arbitrary geodesic metric space $X$ were described in
detail in joint work of the authors with Dodziuk in \cite{BDS}
where we proved the following proposition:

\begin{prop} \label{pulled-string}
The notion of a metric space with a pulled string is 
a metric space $(Y, d_Y)$ constructed from a metric space $(X,d_X)$ 
of Hausdorff dimension $\ge 2$ %NECESSARY OR NOT?
with a curve $C:[0,1]\to X$, so that
\be\label{pulled-string-def1}
Y = X \setminus C[0,1] \disjointunion \{p_0\}, \qquad p_0=C(0),
\ee
where for $x_i \in Y$ we have
\be\label{pulled-string-def2}
	d_Y(x, p_0) = \min \{ d_X(x, C(t)) : \, t\in [0,1]\}
\ee
and for $x_i \in X \setminus C[0,1]$  we have
\be\label{pulled-string-def3}
d_Y(x_1, x_2) =\min\left\{\, d_X(x_1, x_2), \min\{d_X(x_1, C(t_1)) + d_X(x_2, C(t_2)): \, t_i \in [0,1] \}\, \right\}. 
\ee
If $(X,d,T)$ is a Riemannian manifold then $(Y,d,\psi_\#T)$
is an integral current space whose mass measure is the 
Hausdorff measure on $Y$ and
\be\label{pulledcurveHm}
\mathcal{H}_Y^m(Y)=\mathcal{H}_X^m(X)-\mathcal{H}_X^m(K).
\ee
If $(X, d_X, T)$ is an integral current space then $(Y, d_Y, \psi_{\#}T)$ is
also an integral current space where $\psi: X\to Y$ such that $\psi(x)=x$ for all $x\in X\setminus C[0,1]$ and $\psi(C(t))=p_0$ for all $t\in [0,1]$.   So that
\be\label{massofpulledcurve}
\mass(\psi_{\#}T)=\mass(T)    
\ee
\end{prop}

Here we will pull an entire compact set, $K \subset X$, to a point.
In our applications, $X$ will be a Riemannian manifold and $K$
a compact submanifold in the Riemannian manifold.   This is 
described here in the following pair of lemmas proven in \cite{BDS}.
Note that it is only called a pulled string space if $K$ is the image of a curve.

% pulled subset
\begin{lem} \label{pulled-subset-1}
Given a metric space $(X, d_X)$ and a compact set $K \subset X$
we may define a new metric space $(Y, d_Y)$ pulling the set $K$ to a point $p_0 \in K$ by setting 
\be\label{pulled-set-def1}
Y := (X \setminus K) \disjointunion \{p_0\}, \qquad p_0 \in K \, \,\textrm{fixed}, 
\ee
and, for $x \in Y$, we have 
\be\label{pulled-set-def2}
d_Y(x, p_0) = \min\{ d_X(x, y) : \, y\in K\}
\ee
and, for $x_i \in Y \setminus \{p_0\}$, we have
\be\label{pulled-set-def3}
d_Y(x_1, x_2) =\min\left\{ d_X(x_1, x_2), \min\{d_X(x_1, y_1) + d_X(x_2, y_2): \, y_i\in K\} \right\}. 
\ee
\end{lem}

\begin{lem} \label{pulled-subset-2}
If $(X, d_X, T)$ is an integral current space with a compact
subset $K \subset X$ then $(Y, d_Y, \psi_{\#}T)$ is
also an integral current space
where $(Y, d_Y)$ is defined as in Lemma~\ref{pulled-subset-1} and
where $\psi: X\to Y$ such that $\psi(x)=x$ for all $x\in X\setminus K$
and $\psi(q)=p_0$ for all $q\in K$.   In addition
\be\label{massofpulledset}
\mass(\psi_{\#}T)=\mass(T) - ||T||(K)   
\ee
If $(X,d_X,T)$ is a Riemannian manifold then $(Y,d_Y,\psi_\#T)$
is an integral current space whose mass measure is the 
Hausdorff measure on $Y$ and
\be\label{pulledsetHm}
\mathcal{H}_Y^m(Y)=\mathcal{H}_X^m(X)-\mathcal{H}_X^m(K).
\ee
\end{lem}

%%%%%%%%%%%%%%%%%%%%%%%%%%%%%%%
%Sewing to Pulled Strings
\subsection{Scrunching to Pulled Metric Spaces}

In this subsection we generalize a theorem proven by the authors
with Dodziuk in \cite{BDS} 
concerning the limit of a sequence of manifolds which ``scrunch'' 
a compact set to a point as follows:

\begin{defn}\label{def-seq-down}
Given a single Riemannian manifold, $M^3$, with a compact set, 
$A_0\subset M$.  A sequence of manifolds,
\be
N_j^3= (M^3 \setminus A_{\delta_j})\disjointunion A'_{\delta_j} 
\ee
is said to scrunch $A_0$ down to a point
if $A_{\delta}=T_{\delta}(A_0)$ and
$A'_\delta$ satisfies:
\be\label{sewn-curve-tubular'}
%(1-\epsilon)\vol(A_{\delta}) \le % REMOVED: when we do applications with scrunching a region with positive volume 
\vol(A_{\delta}')\le \vol(A_{\delta})(1+\epsilon) 
\ee
and
\be\label{sewn-curve-vol'}
%(1-\epsilon) \vol(M^3)\le  % REMOVED: when we do applications with scrunching a region with positive volume 
\vol(N^3) \le \vol(M^3) (1+\epsilon)
\ee
and
\be \label{sewn-curve-diam'}
\diam(A_{\delta}')\le H
\ee
where $\epsilon=\epsilon_j \to 0$ and where $H=H_j \to 0$ and $2\delta_j<H_j$.

The region $A_{0}'$ is referred to as the ``edited region'' constructed via tunnel surgeries and is explicitly detailed in Section~\ref{S:method1-fxdmfld} below.  
\end{defn}

We prove the following new theorem:   

\begin{thm}\label{thm-seq-scrunch}
The sequence $N_j^3$ as in Definition~\ref{def-seq-down} 
where $M^3$ is taken to be compact 
and $A_{0}$ a compact, embedded submanifold of dimension 1 to 3
converges in the
Gromov-Hausdorff sense, 
\be
N_j^{3} \GHto N_\infty,
\ee
and the intrinsic flat sense,
\be
N_j^{3} \Fto N_\infty,     
\ee
where $N_\infty$ is the metric space created by taking $M^3$ and 
pulling $A_0$ to a point $p_0$
as in Lemmas~\ref{pulled-subset-1}-~\ref{pulled-subset-2}.

If, in addition, $\Hm^{3}(A_{0})=0$ then we also have convergence in the metric measure sense
\be
N_j^{3} \mGHto N_\infty.
\ee

\end{thm}

Note that when $A_0$ is the image of a curve, then $N_\infty$ is a pulled thread space as in Remark~\ref{pulled-string}.   In \cite{BDS}, the
authors and Dodziuk proved 
Theorem~\ref{thm-seq-scrunch} in that special case only with stronger consequences that only hold in that setting.

Within this proof we will state four lemmas proven by the
authors with Dodziuk in \cite{BDS}.   We state them because
we will apply them again later in the paper.  We state them within the proof
so that we may motivate and explain them.

\begin{proof}
One first constructs a map $F_j: N_j \to N_\infty$ 
which is the identity away from the region containing all
the tunnels and maps the entire regions containing the tunnels to a single 
point.  More precisely one applies the following lemma from \cite{BDS}:

% lemma surj maps
\begin{lem}\label{L:surjectivemaps}
Given $M^3$ a compact Riemannian manifold (possibly with boundary) and a smooth embedded compact zero to three 
dimensional submanifold $A_0\subset M^3$ (possibly with boundary),
and $N_j$ as in Definition~\ref{def-seq-down}.
Then for $j$ sufficiently large
there exist surjective Lipschitz maps
\be\label{surjlipmaps4}
F_j: N_j^3 \to N_\infty \textrm{ with } \Lip(F_j) \le 4
\ee
where $N_\infty$ is the metric space created by taking $M^3$ and 
pulling $A_0$ to a point $p_0$
as in Lemmas~\ref{pulled-subset-1}-~\ref{pulled-subset-2}. 
\end{lem}

In \cite{Gromov-metric}, Gromov proved that
the Gromov-Hausdorff distance between metric spaces
satisfies,
\be
d_{GH}(N_j, N_\infty) \le 2H_j,
\ee
if there is a map $F_j: N_j^3 \to N_\infty$ which is an
$H_j$-almost isometry:
\be
|d_{N_\infty}(F_j(p), F_j(q)) - d_{N_j}(p,q)| \le H_j
\ee
and
\be 
\forall y\in N_\infty \,\,\,\exists x\in N_j\textrm{ such that }
d_{N_\infty}(F_\infty(x),y) \le H_j.
\ee
The converse is also true requiring again a factor of two:
if there is an $H_j$-almost isometry then the Gromov-Hausdorff
distance between the spaces is $\le 2H_j$.  Applying this
we have the following lemma proven in \cite{BDS}:

\begin{lem}\label{L:almostiso}
Given $N_j^{3}$ as in Definition~\ref{def-seq-down},
the maps $F_j: N_j^3 \to N_\infty$ of (\ref{surjlipmaps4})
are $H_j$-almost isometries with $\lim_{j\to \infty}H_j=0$.
Thus 
\be\label{NjGHtoNinfty}
N_j \GHto N_\infty.
\ee
\end{lem}

In order to prove metric measure convergence one needs only to show
$F_j$ push forward the measures to measures that converge to the measure
on the limit space.  This only works when $A_0$ has measure $0$.
One obtains this in the following lemma proven in \cite{BDS}:

\begin{lem}\label{L:mGHconv}
Given $N_j^3 \to N_\infty$ as in Lemma~\ref{L:surjectivemaps}
endowed with the Hausdorff measures, then 
we have metric measure convergence 
if $A_0$ has $\mathcal{H}^{3}$-measure $0$.
\end{lem}

Whenever one has Gromov-Hausdorff convergence, it was proven by the second author with Wenger in \cite{SW-JDG} that a subsequence converges to
an intrinsic flat limit lying within the Gromov-Hausdorff limit.  However due to collapse or cancellation the limits might not agree.  The intrinsic flat limit is
always a rectifiable metric space lying within the Gromov-Hausdorff limit.
In order to prove intrinsic flat convergence one may apply a
theorem of the second author proven in \cite{Sormani-AA}.  This was
done to show the following lemma. %proven in \cite{BDS}.  

\begin{lem}\label{L:flatconv}
Let $N_j^3$ be exactly as in Lemma~\ref{L:surjectivemaps} and Lemma~\ref{L:almostiso} 
where we assume $M^{3}$ is compact and we have a compact set,
$A_0\subset M\setminus \partial M$.
Then there exists an integral current space $N$ such that $\bar{N}$ is isometric to $N_\infty$ and 
\be
N_j \Fto N_{\infty}, [FIXED]
\ee
where $N_{\infty}$ is given the integral current as in Lemma~\ref{pulled-subset-2},
and when $A_0$ has $\mathcal{H}^{3}$-measure $0$,
\be \label{L:massconv}
\mass(N_j) \to \mass(N)=\mathcal{H}^3(N).   
\ee
\end{lem}

This was proven in \cite{BDS} in the case when $A_{0}$ was a curve. The only fact that needs to be checked in the present case is that $N=\bar{N}$, or that $p_{0}$ has positive density. 

% proof that N=N_{\infty}
Since 
\be
\liminf_{r\to 0} \frac{\vol_{N_\infty}(B(p_0,r))}{r^3}
=\liminf_{r\to 0} \frac{\vol_{M}(T_r(A_0)\setminus A_0)}{r^3}
\ee
Thus $N$ is isometric to $N_\infty$ when this
liminf is positive and $N$ is isometric to $N_\infty\setminus \{p_0\}$
when this liminf is $0$. Since $A_{0}$ is a compact, embedded submanifold of dimension $d=1,2,3$ in a 3 dimensional Riemannian manifold we have
\be
 \liminf_{r\to 0} \frac{\vol_{M}(T_r(A_0)\setminus A_0)}{r^3}
 = \liminf_{r\to 0} \frac{\omega_{d} r^{3-d} \cdot \Hm^{d}(A_{0})}{r^3} = + \infty >0,
\ee
since $\Hm^{d}(A_{0})$ is never zero. Thus $N$ is isometric to $N_\infty$.

Applying this final lemma the proof of Theorem~\ref{thm-seq-scrunch}
is complete.
\end{proof}

\subsection{Remarks on Convergence to Pulled Limit Spaces}

Sequences of Riemannian manifolds with sectional curvature
uniformly bounded below converge in the Gromov-Hausdorff and Intrinsic Flat
sense to Alexandrov spaces with curvature bounded below 
(cf. \cite{BBI}).  It can be seen  that a sequence of $N_j$ as in 
Definition~\ref{def-seq-down} that also
have a uniform lower bound on sectional cannot exist because
if they did they would have a limit which is a pulled space that
fails to have Alexandrov curvature bounded below. 

Noncollapsing sequences of Riemannian manifolds with Ricci curvature
uniformly bounded below converge in the Gromov-Hausdorff and Intrinsic Flat
sense to spaces whose Hausdorff measure satisfies the appropriate
Bishop-Gromov Volume Comparison Theorem. 
It can be shown that that a noncollapsing 
sequence of $N_j$ as in Definition~\ref{def-seq-down} that also
have a uniform lower bound on Ricci curvature cannot exist because
if they did they would have a limit which is a pulled space 
whose Hausdorff measure fails to
satisfy the Bishop-Gromov Volume Comparison Theorem. 
 
However we can
construct sequences of manifolds with positive scalar curvature
that converge to a wide variety of interesting pulled limit spaces.

\section{Method I: Sewing a Fixed Manifold}\label{S:method1-fxdmfld}

This first method takes a fixed Riemannian manifold, $M^3$, 
with positive scalar curvature that has a region with constant positive
sectional curvature $K$ about compact set $A_{0}$.   We then construct a
sequence of Riemannian manifolds, $N_j^3$, that also have positive scalar curvature that converge to a pulled space constructed from $M^3$
by pulling the compact set $A_{0}$ to a point.   The process which we call ``sewing a region" involves constructing many tunnels between
many points in the region.

We break this section into three parts: first we describe a typical tunnel,
then we describe how to glue a whole collection of tiny tunnels
into a manifold, we do this in a very specific way to sew the region,
and finally we prove convergence of the sequence to the pulled limit.  

\subsection{Tunnels with Positive Scalar Curvature}\label{sect-tunnel-back}

%We ought to explain a bit about the Schoen-Yau construction vs the Gromov-Lawson one.

Using different techniques, Gromov-Lawson and Schoen-Yau described how to
construct tunnels diffeomorphic to ${\mathbb{S}}^2 \times [0,1]$
with metric tensors of positive scalar curvature that can be glued smoothly
into three dimensional spheres of constant sectional curvature \cite{Gromov-Lawson-tunnels}\cite{Schoen-Yau-tunnels}.  One may imagine constructing such a tunnel by taking the Schwarzschild
Riemannian manifold, dropping a tangent sphere from above and raising one
from below, and then smoothing in a way such that the scalar curvature becomes positive.   See Figure~\ref{F:tunnels}. 
These tunnels are the first crucial piece for our construction.

\begin{figure}[h] % FIGURE TUNNEL U
\begin{center}
	\includegraphics[width=2in]{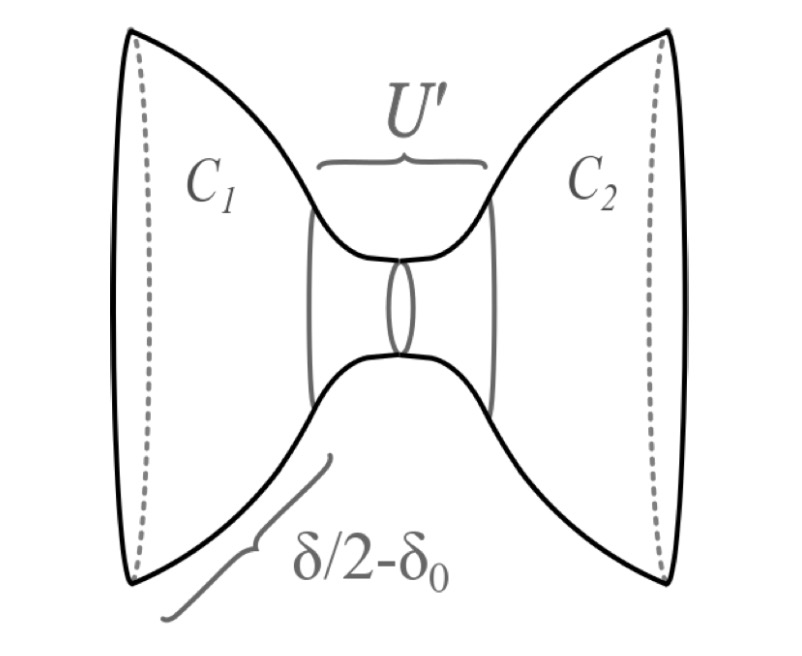}
\end{center}
%	\caption{The tunnels $U'$ and $U$, respectively.}\label{F:tunnels}
	\caption{The tunnel $U$.}\label{F:tunnels}
\end{figure}

These tunnels can be made long or short, or arbitrarily tiny.  Tiny ones are rigorously constructed by the first author with J. Dodziuk in the appendix to \cite{BDS}.  In their work, the tunnels were used to bridge between two regions 
within a single manifold where the regions are isometric
to convex balls in round three spheres.  To be more precise, they have proven Lemma~\ref{tunnellemma} which we restate below.  Note also Remark~\ref{min-scal} following the statement.

% BEGIN LEMMA STATEMENT
\begin{lem}\label{tunnellemma}
	Let $0<\delta/2 < 1$. % ADDED 3-31-16 
	Given a complete Riemannian manifold, $M^3$, 
	that contains two balls $B(p_i,\delta/2)\subset M^3$, $i=1,2$, with constant positive sectional curvature $K \in (0,1]$ on the balls, 
	and given any $\epsilon>0$, there exists a $\delta_0>0$ sufficiently small so that we may create a new
	complete Riemannian manifold, $N^3$, 
	in which we remove two balls and glue in a cylindrical region, $U$, between them:
	\be\label{TL-sewnN}
	N^3=M^3 \setminus \left(B(p_1,\delta/2)\cup B(p_2,\delta/2)\right) \disjointunion U
	\ee
	where $U=U(\delta_0)$ has a metric of positive scalar curvature with
	\be\label{TL-diameterU}
	\diam(U) \le h=h(\delta), 
	\ee
	where
	\be\label{TL-tunnellengthtozero}
	\lim_{\delta\to 0} h(\delta)=0 
	\textrm{ uniformly for } K\in (0,1].
	\ee
	The collars $C_i= B(p_i,\delta/2) \setminus B(p_i,\delta_0)$ identified with subsets of
	$N^3$ have the original metric of constant curvature and the tunnel $U'=U\setminus (C_1\cup C_2)$ has arbitrarily small diameter $O(\delta_0)$ and volume $O(\delta_0^3)$. 
	Therefore with appropriate choice of $\delta_0$, we have
	\be\label{TL-volumeestU}
	(1-\epsilon) 2\vol(B(p,\delta/2)) \le \vol(U) \le (1+\epsilon) 2\vol(B(p,\delta/2))
	\ee
	and
	\be\label{TL-volumeestN}
	(1-\epsilon) \vol(M) \le \vol(N) \le (1+\epsilon) \vol(M).
	\ee
\end{lem}

\begin{rmrk} 
After inserting the tunnel,
$\partial B(p_1,\delta/2)$ and $\partial B(p_2,\delta/2)$ are arbitrarily close together because of (\ref{TL-tunnellengthtozero}).
\end{rmrk}

\begin{rmrk}  \label{min-scal}
We note that since the scalar curvature inside the tunnel is positive we have
the following fact applied by the authors and Dodziuk in \cite{BDS}:  
\be
\Scal_p >0 \,\,\forall p \in M \quad \implies \quad
\Scal_p >0 \,\,\forall p \in N.
\ee
  In this paper we will also apply the fact:
\be
\Scal_p \ge -\varepsilon \,\,\forall p \in M \quad \implies \quad
\Scal_p \ge -\varepsilon \,\,\forall p \in N.
\ee
\end{rmrk}

\subsection{Gluing tunnels into a Fixed Manifold}

In \cite{BDS}, the authors and Dodziuk described a process of altering
a manifold with positive scalar curvature, $M$, to
build a {\em sewn manifold}, $N$.  This process called {\em sewing along a curve}, $C:[0,1] \to M$, involved cutting out a sequence of balls about 
carefully chosen sequential
points, $C(t_i)$, along the curve and replacing them with tunnels running
from the sphere about $C(t_{2i})$ to the sphere about $C(t_{2i+1})$.
The first step in the construction was the following proposition which
we will apply again in this paper to a completely different collection of
points to create our new method of construction:

\begin{prop}\label{prop-glue}  %used to be called prop-sewn 
Given a complete Riemannian manifold, $M^3$, and $A_{0} \subset M^3$ a compact subset with an even number of points $p_{i} \in A_{0}$, $i = 1, \ldots, n$, with pairwise disjoint contractible balls $B(p_i,\delta)$ which have constant positive sectional curvature $K$, for some $\delta>0$, define $A_{\delta} = T_{\delta}(A_{0})$ and
\be\label{prop-glue-defA'}
	A_{\delta}' = A_{\delta} \setminus \left( \bigcup_{i=1}^n B(p_i,\delta/2) \right) 	
		\disjointunion \bigcup_{i=1}^{n/2} U_i
\ee
where $U_i$ are the tunnels as in Lemma~\ref{tunnellemma} connecting $\partial B(p_{2j+1},\delta/2)$ to $\partial B(p_{2j+2},\delta/2)$ for $j=0,1,\ldots,n/2-1$. 
Then
given any $\epsilon>0$, shrinking $\delta$ further, if necessary, we may create a new complete Riemannian manifold, $N^3$, 
\be\label{E:prop-glue}
	N^3 = (M^3 \setminus A_{\delta}) \disjointunion A_{\delta}'
\ee
satisfying
\be\label{prop-glue-volA'}
(1-\epsilon)\vol(A_{\delta}) \le \vol(A_{\delta}')\le \vol(A_{\delta})(1+\epsilon)
\ee
and
\be\label{prop-glue-vol-space}
(1-\epsilon)\vol(M^3)\le \vol(N^3) \le \vol(M^3) (1+\epsilon).
\ee

If, in addition, $M^3$ has non-negative or positive scalar curvature, then so does $N^3$.
In fact,
\be \label{inf-scal1}
\inf_{x\in M^3} \Scal_x \ge \min \left\{0, \inf_{x\in N^3} \Scal_x\right\}
\ee
If $\partial M^3 \neq \emptyset$, the balls avoid the boundary
and $\partial M^3$ is isometric to $\partial N^3$.
\end{prop}

\bd % def sewn manifold
We say that we have glued the manifold to itself with a tunnel between the collection of pairs of sphere  $\partial B(p_i,\delta)$ to $\partial B(p_{i+1},\delta)$
for $i=1$ to $n-1$.   
\ed

%%%%%%%%%%%%%%%%%%
\subsection{Sewing Compact Sets}

Here we introduce the notion of sewing a compact set in a manifold.  This is very different from the notion of sewing along a curve that was introduced by the
authors with Dodziuk in \cite{BDS}.  In both one constructs a collection of tunnels in the space using Proposition~\ref{prop-glue}, however, when one sews along a curve the balls are simply lined up along the curve.  To sew a region the picture is much more complicated.

The goal is to scrunch the region to a point as in Definition~\ref{def-seq-down}.
So we need to use the tunnels to bring every point in the region close to any other
point in the region.   Before stating and proving our proposition we
describe the key idea with a figure.

We start with a compact subset $A_{0}$ of $M^{3}$ with a tubular neighborhood that is isometric to a compact subset of a sphere with constant sectional curvature. 
We cover the tubular neighborhood 
with $n$ balls of small equal radius, and consider the disjoint
collection of balls of radius $r$ about the same points.  Every point in the
tubular neighborhood is close to one of the $r$-balls.

In every one of the $r$-balls we cut out $(n-1)$ tiny disjoint balls of radius $\delta$.  
We then glue tunnels between these tiny $\delta$-balls so that there is a 
tunnel running from any small $r$-ball to any other small $r$-ball
using Proposition~\ref{prop-glue}.
  See Figure~\ref{fig-sewn-region} which has $n=3$ $r$-balls and $n(n-1)=6$ $\delta$-balls and $n(n-1)/2=3$ tunnels.

\begin{figure}[htbp]
\begin{center}
\includegraphics[width=3in]{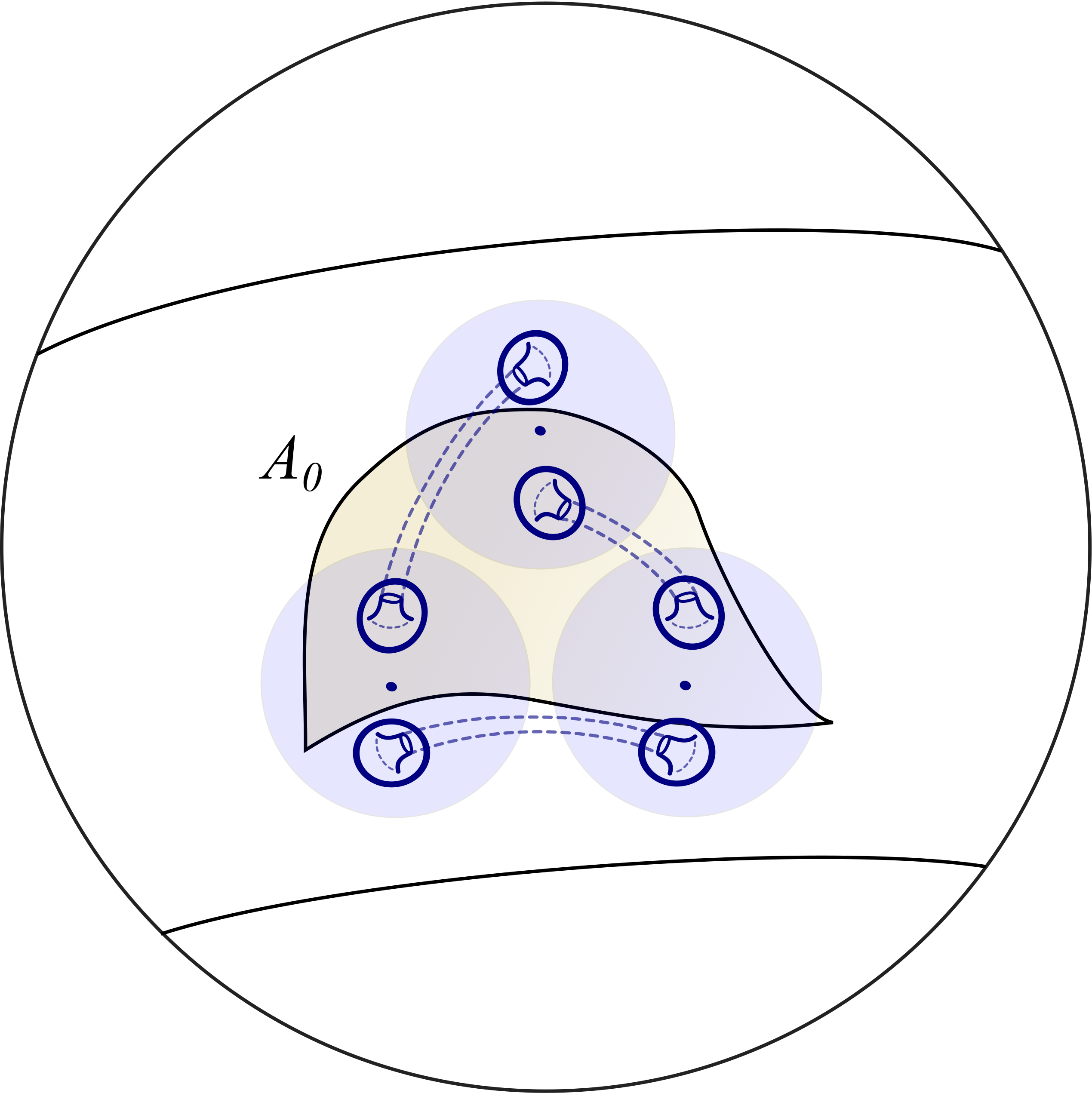}
\caption{Sewing a manifold through six balls along a region.}
\label{fig-sewn-region}
\end{center}
\end{figure}

Once we have done this sewing we will have created a new sewn
manifold $N^3$ with $n(n-1)/2$ tunnels.   One may think 
of this $N^3$ as being a new version of $M^3$ with a collection of star
gates, providing quick service from anywhere in the special region $A_{0}$ 
to anywhere else in the special region.  To draw a minimal path from $p$ to $q$, 
one runs the path through the
tunnel between the $r$-ball closest to $p$ and the $r$-ball closest
to $q$.  Away from the special region $A_{0}$, $N^3$ is isometric to $M^3$
but it is scrunched on that region.

\begin{prop}\label{sewn-region}
Given a complete Riemannian manifold, $M^3$, and 
a compact set $A_{0}\subset M^3$ whose tubular neighborhood,
$A_{a} = T_{a}(A_{0})$, is Riemannian isometric to 
a subset of a sphere of constant sectional curvature.

Let $r\in (0,a)$.  Given $\epsilon>0$, there exists $\delta = \delta(A_{0},K,r,\epsilon)\in (0,r)$ and there exists even $n=\bar{n}(\bar{n}-1)$ depending on $A_{0}, K$, $\epsilon$, and $r$ and points $p_1,...,p_n \in A_{0}$ with $B(p_i,\delta)$
are pairwise disjoint such that we can ``sew the region tightly'' to create a new complete Riemannian manifold $N^3$,
\be
N^3 = (M^3 \setminus A_{r} )\disjointunion A_{r}',
\ee
exactly as in Proposition~\ref{prop-glue}, with
\be
A_{r}'= A_{r} \setminus \left(\bigcup_{i=1}^{n} B(p_i,\delta/2)\right)\disjointunion \bigcup_{i=1}^{n/2} U_i, 
\ee
so that
\be\label{sewn-region-tubular}
%(1-\epsilon)\vol(A_{r}) \le  REMOVED
\vol(A_{r}')\le \vol(A_{r})(1+\epsilon) 
\ee
and
\be\label{sewn-region-vol}
%(1-\epsilon) \vol(M^3)\le  REMOVED
\vol(N^3) \le \vol(M^3) (1+\epsilon)
\ee
and
\be \label{sewn-region-diam}
\diam(A_{r}')\le H(r,\delta)= 16r+ 3h(\delta). 
\ee

Since $\delta \to 0$ when $r \to 0$, 
\be\label{sewn-region-arbitrarily-tightly}
\lim_{r \to 0} H(r,\delta)=0 \textrm{ uniformly for } K\in (0,1], 
\ee
we say we have sewn the region $A_{0}$ arbitrarily tightly. 

If $M^3$ has non-negative or positive scalar curvature, then so does $N^3$.
In fact,
\be \label{inf-scal-2}
\inf_{x\in M^3} \Scal_x \ge \min \left\{0, \inf_{x\in N^3} \Scal_x\right\}
\ee
If $\partial M^3 \neq \emptyset$, the balls avoid the boundary
and $\partial M^3$ is isometric to $\partial N^3$.   
\end{prop}

\begin{proof}
Fix $r<a$ as in the proposition statement.
For simplicity of notation, let $A=A_{r}$ and $A'=A_{r}'$. 

By the compactness of $A_{0}$ there exists a finite $\bar{n}=\bar{n}(A_{0},r)$ equal to the maximal number of pairwise disjoint balls $\{B(v_k,r)\}_{k=1}^{\bar{n}}$, centered at $v_k \in A_{0}$ of radius $r>0$.   Note that $B(v_k,r) \subset A$. 

Let $\delta=\delta(\bar{n},K,r)>0$ be chosen small enough so that for each $k=1$ to $\bar{n}$, there are $\bar{n}-1$ pairwise disjoint balls of radius $\delta$ centered at $v_{kj}$ with $k\neq j$ such that %note $k\neq j$ is crucial
\be
B(v_{kj},\delta)\subset B(v_k,r) \textrm{ for } j\in \{1,...,\bar{n}\}\setminus \{k\}
\ee
and each 
\be\label{vkj-edge}
v_{kj} \in \partial B(v_k,r-\delta).   
\ee
Let
\be
	n=\bar{n}(\bar{n}-1).
\ee
We choose the points $p_i\in A_{0}$ such that
\be
\{p_1,....,p_n\}= \{v_{kj}:\,\, k,j\in \{1, \ldots, \bar{n}\},\, k\neq j\}
\ee
so that $B(p_i,\delta)$ are disjoint balls centered in $A_{0}$.
In fact we choose $p_i$ so that when $i$ is even we have both
\be
p_{i} \in \{v_{kj}:\,\, k<j\} 
\ee
and
\be\label{switch}
p_{i+1}=v_{jk} \textrm { iff } p_{i}=v_{kj}
\ee
and we set
\be
U_{kj}=U_{j,k}=U_i.
\ee

We now apply Proposition~\ref{prop-glue} with
\begin{eqnarray}
A'&=& A \setminus \left(\bigcup_{i=1}^{n} B(p_i,\delta/2)\right)\disjointunion \bigcup_{i=1}^{n/2} U_i\\
&=&A \setminus \left(\bigcup_{k\neq j} B(v_{kj},\delta/2)\right)\disjointunion \bigcup_{k<j} U_{kj}
\end{eqnarray}
to conclude the volume estimates (\ref{sewn-region-tubular}) and (\ref{sewn-region-vol}). 

We next verify the diameter estimate of $A'$, (\ref{sewn-region-diam}).
To do this we define sets $C_k \subset A'$ which correspond to the
sets $\partial B(v_k,r)\subset A$ which are unchanged because the 
$B(p_i,\delta/2)\cap \partial B(v_k,r)=\emptyset$.   
We also define sets $C_{kj} \subset A'$ which correspond to the sets $\partial B(v_{kj},\delta/2)\subset A$ which are unchanged because they
are the edges of the edited regions:  
\be
C_{kj}\cup C_{jk} =\partial U_{jk}. 
\ee
Let
\be
U=\bigcup_{k<j} U_{kj}.
\ee

Let $x$ and $y$ be arbitrary points in $A'$.   We first claim there exists
$k,j$ such that
\be\label{a+4r+h}
d_{A'}(x, C_k)< 4r+h(\delta)
\textrm{ and } d_{A'}(y, C_j) < 4r+h(\delta).
\ee
By symmetry we need only prove this for $x$.   Note that in Case I where
\be
x \in A'\setminus \bigcup_{k<j} U_{kj} = A \setminus \left(\bigcup_{k\neq j} B(v_{kj},\delta/2)\right)
\ee
then we can view $x$ as a point in $A$.   Let $\gamma_1\subset A$ be
the shortest path from $x$ to the closest point $v_x\in A_{0}$, then $L(\gamma_1)<r$.  If
\be\label{E:claimgammamissball1}
\gamma_1 \cap B(v_{kj},\delta/2)\neq \emptyset
\ee
then there exists $k$ such that
\be\label{E:claimdistlessthana}
d_{A'\setminus U}(x, C_{kj}) <\delta
\ee
and we have
\begin{eqnarray}
d_{A'\setminus U}(x,C_k) &\le& d_{A'\setminus U}(x,C_{kj}) 
	+ \diam(C_{kj}) + d_{A'\setminus U}(C_{kj},C_k) \\
	&& \qquad \textrm{(by the triangle inequality)}\\
	&<& r + \diam(C_{kj}) + \delta/2
	 	\qquad \textrm{(by (\ref{E:claimdistlessthana}) and (\ref{vkj-edge}))}\\
	&\le& r + \pi(\delta/2) + \delta/2 
		\qquad \textrm{(by sec $\equiv K >0$ on $B(v_{kj},\delta)$)}\\
	&<& r+3\delta,
	\end{eqnarray}
so that (\ref{a+4r+h}) holds.   Otherwise, still in Case I, if (\ref{E:claimgammamissball1}) fails, $\gamma_1\subset A'\setminus U$ and $\gamma_1(1)=v_x \in A_{0}$.   Let $\gamma_2\subset A$ be the shortest path from $v_x$ to the nearest $\partial B(v_{k'},r)$.   Then $L(\gamma_2)< r$ because this
was a maximal collection of disjoint balls.   If 
\be\label{E:claimgammamissball2}
\gamma_2 \cap B(v_{kj},\delta/2)\neq \emptyset
\ee
then there exists $k$ such that
\be\label{E:claimdistlessthanr}
d_{A'\setminus U}(v_x, C_{kj}) <r
\ee
and we have
\begin{eqnarray}
d_{A'\setminus U}(v_x,C_k) &\le& d_{A'\setminus U}(v_x,C_{kj}) 
	+ \diam(C_{kj}) + d_{A'\setminus U}(C_{kj},C_k) \\
	&& \qquad \textrm{(by the triangle inequality)}\\
	&<& (r-\delta/2) + \diam(C_{kj}) + \delta/2 
	 	\qquad \textrm{(by (\ref{E:claimdistlessthanr}) and (\ref{vkj-edge}))}\\
	&\le& (r-\delta/2) + \pi(\delta/2) + \delta/2 \\
	&&	\qquad \textrm{(by sec $\equiv K >0$ on $B(v_{kj},\delta)$)}\\
	&<& r+2\delta
\end{eqnarray}
so, because $\delta<r$,
\be
d_{A'\setminus U}(x, C_k) < L(\gamma_1)+3r<4r
\ee
and we have (\ref{a+4r+h}).   Otherwise, still in Case I but when (\ref{E:claimgammamissball2}) fails, 
\be
d_{A'\setminus U}(x, C_{k'}) < L(\gamma_1)+L(\gamma_2)<2r.
\ee
Alternatively, we have Case II where
\be
x\in \bigcup_{k<j} U_{kj}.
\ee
In this case, there exists a $k$ such that $x\in U_{kj}$ and
\be
d_{A'}(x, C_k)\le \diam(U_{kj}) + \dist_{A'}(C_{kj},C_k) \le h(\delta)+\delta/2
	< h(\delta)+r.
\ee
Thus we have the claim in (\ref{a+4r+h}).

We now proceed to prove (\ref{sewn-region-diam}) by estimating $d_{A'}(x,y)$ for $x,y \in A'$.  If $j=k$ in (\ref{a+4r+h}), then $d_{A'}(x,y) \le 2(4r+h(\delta))$.   Otherwise,
\be\label{E:diamalmostdone1}
d_{A'}(x,y) \le 2(4r+h(\delta)) + \sup \{ d_{A'}(z,w) \mid z 
\in C_k, w \in C_j \} 
\ee
and
\begin{eqnarray}
\sup \{ d_{A'}(z,w) \mid z \in C_k, w \in C_j \} &\le&
	\diam_{C_k}(C_k) + \dist(C_k,C_{kj}) + \diam(U_{kj}) \notag\\
&&\qquad	+ \dist(C_{jk},C_j)+\diam(C_j) \\
&\le& \pi r + \delta/2 + h(\delta) + \delta/2 + \pi r \\
&\le& 8r+h(\delta).\label{E:diamalmostdone2}
\end{eqnarray}
Thus, by (\ref{E:diamalmostdone1}) and (\ref{E:diamalmostdone2}) we have 
\be
d_{A'}(x,y) \le 8r+2h(\delta)+8r+h(\delta) \le 16r+ 3h(\delta),
\ee
which is the desired diameter estimate (\ref{sewn-region-diam}).

We observe that by our choice of $\delta$ satisfying $\delta<r$ and the fact that $h(\delta) = O(\delta)$ from Lemma~\ref{tunnellemma} we have that (\ref{sewn-region-arbitrarily-tightly}) holds. 

Finally, observe that (\ref{inf-scal-2}) follows since Lemma~\ref{tunnellemma} shows that the tunnels $U_{i}$ have positive scalar curvature. 
\end{proof}

\subsection{Sewing a Fixed Manifold to a Pulled Limit Space}

We may now describe our first technique for creating a sequence of
manifolds with positive scalar curvature which converge to a pulled
metric space.  We will start with a fixed Riemannian manifold $M^3$ 
and create an increasingly tightly sewn sequence of manifolds using
increasingly dense collections of increasingly tiny balls pairwise connected by
increasingly tiny tunnels as in Definition~\ref{def-seq}.   We then prove
this sequence converges to a pulled metric space in Theorem~\ref{thm-seq-sewn}.

\begin{defn}\label{def-seq}
Given a single compact Riemannian manifold, $M^3$, with a compact set, 
$A_0\subset M$, with a tubular neighborhood
$A=T_a(A_0)$ which is Riemannian isometric to a tubular neighborhood of a compact set $V \subset \mathbb{S}^3_K$, in a 
standard sphere of constant sectional curvature $K$,
satisfying the hypothesis of Proposition~\ref{sewn-region}.
We can construct its sequence of increasingly tightly sewn
manifolds, $N_j^3$, by applying Proposition~\ref{sewn-region} taking 
$\epsilon=\epsilon_j \to 0$, $n=n_j \to \infty$, and $\delta=\delta_j\to 0$ to create each sewn manifold, $N^3=N_j^3$ and the edited regions $A_{\delta}'=A_{\delta_{j}}'$ which we simply denote by $A_{j}'$.   
Since these sequences $N_j^3$ are created using Proposition~\ref{sewn-region}, they have nonnegative (resp. positive) scalar curvature whenever $M^3$ has nonnegative (resp. positive) scalar curvature
and $\partial N_j^3=\partial M^3$ whenever $M^3$ has a
nonempty boundary.
\end{defn}

Note that by Proposition~\ref{sewn-region}, the $N_j^{3}$ in this sequence are scrunching the compact set
$A_0$ to a point as in Definition~\ref{def-seq-down}.
Thus our new Theorem~\ref{thm-seq-scrunch} immediately implies the following theorem: 

\begin{thm}\label{thm-seq-sewn}
The sequence $N_j^3$, as in Definition~\ref{def-seq} assuming $M^{3}$ is compactand  $A_{0}$ is a compact, embedded submanifold of dimension 1 to 3, converges in the Gromov-Hausdorff sense 
\be
N_j^{3} \GHto N_\infty,
\ee
and the intrinsic flat sense
\be
N_j^{3} \Fto N_\infty,
\ee
where $N_\infty$ is the metric space created by pulling the
region $A_0 \subset M$ to a point as in Lemmas~\ref{pulled-subset-1}-~\ref{pulled-subset-2}.

If, in addition, $\Hm^{3}(A_{0})=0$ then we also have convergence in the metric measure sense
\be
N_j^{3} \mGHto N_\infty.
\ee
\end{thm}

A special case of this theorem appeared in \cite{BDS} where the authors proved that a sequence of manifolds sewn increasingly tightly along the image of a curve converged to a pulled string space.   

\section{Sewn Spheres and Limits of Volumes}\label{S:sewn-spheres} 

In this section we apply Method I to a standard sphere $M^{3}=\mathbb{S}^{3}$ of constant curvature one with $A_{0}$ chosen to be either a closed geodesic or an equatorial 2-sphere.

The first example appeared in work of the authors with Dodziuk
proven by sewing along a curve \cite{BDS}.  

\begin{example}\label{sphere-geod} \cite{BDS}
Let $N_j^3$ be the sequence of manifolds with positive scalar curvature
constructed from the standard sphere, $\mathbb{S}^3$, by sewing along a closed geodesic $C:[0,1]\to \mathbb{S}^3$
with $\delta=\delta_j \to 0$.   Then
\be\label{sphere-geod-conv}
N_j^3 \mGHto N_\infty \textrm{ and } N_j^3 \Fto N_{\infty} %\bar{N}
\ee
where $N_\infty$ is the metric  space created by taking the
standard sphere and pulling the geodesic to a point as in Proposition~\ref{pulled-string}.
%and $\bar{N}$ is isometric to $N_\infty$.   

Moreover, at the pulled point $p_0\in N_{\infty}$ we have $w\Scal(p_{0})=-\infty$ where $w\Scal$ is as in (\ref{E:scalarlimit}), i.e.
\be\label{E:scalarbad} % SCALAR BAD
w\Scal(p_{0}) = \lim_{r \to 0^{+}} 30\left(\frac{\vol_{\E^{3}}(B(0,r)) - \Hm^{3}(B(p,r))}{r^{2}\cdot \vol_{\E^{3}}(B(0,r))} \right) = - \infty.
\ee
\end{example}

This example is depicted in Figure~\ref{fig-sewn-sphere}.

\begin{figure}[htbp]
\begin{center}
\includegraphics[scale=0.3]{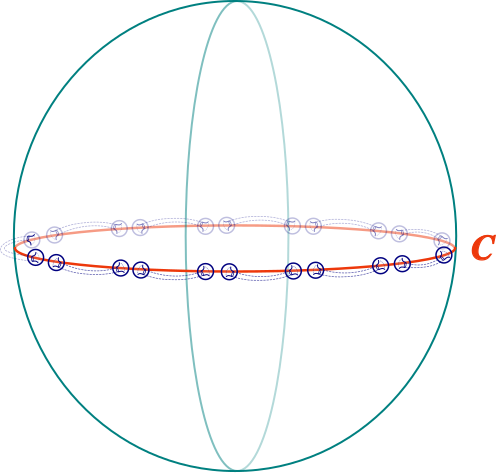}
\caption{Sewing along a closed geodesic in a sphere $\mathbb{S}^{3}$.}
\label{fig-sewn-sphere}
\end{center}
\end{figure}

Similarly, when scrunching a 2-sphere inside $\mathbb{S}^{3}$:

\begin{example}\label{sphere-equator}
Let $N_j^3$ be the sequence of manifolds with positive scalar curvature
constructed from the standard sphere, $\mathbb{S}^3$, by sewing along an equatorial 2-sphere, $A_{0} =\mathbb{S}_e^2 \subset \mathbb{S}^3$, with $\delta=\delta_j \to 0$ as in Proposition~\ref{sewn-region}.   Then
\be\label{sphere-equator-conv}
N_j^3 \mGHto N_\infty \textrm{ and } N_j^3 \Fto N_{\infty} %\bar{N}
\ee
where $N_\infty$ is the integral current space created by taking the
standard sphere and pulling the equatorial sphere to a point as in 
Lemma~\ref{pulled-subset-1}.
% and $\bar{N}$ is isometric to $N_\infty$.   

Moreover, at the pulled point $p_0\in N_{\infty}$ we have $w\Scal(p_{0})=-\infty$ where $w\Scal$ is as in (\ref{E:scalarlimit}), i.e.
\be\label{E:scalarbad2} % SCALAR BAD
w\Scal(p_{0}) = \lim_{r \to 0^{+}} 30\left(\frac{\vol_{\E^{3}}(B(0,r)) - \Hm^{3}(B(p,r))}{r^{2}\cdot \vol_{\E^{3}}(B(0,r))} \right) = - \infty.
\ee
\end{example}

\begin{proof}

First, observe that 
\begin{eqnarray}
\vol_{N_{\infty}}(B(p_0,r)) &=& \mathcal{H}_{N_\infty}^3(B(p_0,r)) \\%\quad \textrm{(viewed as a subset of $N_\infty$)}\\
	&=& \mathcal{H}_{N_\infty}^3(B(p_0,r) \setminus \{p_0\}) \\
	&=& \mathcal{H}_{\mathbb{S}^3}^3(T_r(\mathbb{S}_e^2)).
	\label{ballequalstubularnbhd}% \quad \textrm{(viewed as a subset of $\mathbb{S}^3$)}\label{ballequalstubularnbhd}
\end{eqnarray}

Since $\mathbb{S}_{e}^{2}$ is a closed equatorial sphere of area $4\pi$
in a three dimensional sphere, we have
\begin{eqnarray}
\lim_{r\to0} \frac{\mathcal{H}_{\mathbb{S}^3}^3(T_{r}(\mathbb{S}_e^2))}{4\pi (2r) } = 1.
\end{eqnarray}
Thus
\be
\lim_{r\to 0} \frac{\vol_{\E^3}(B(0,r)) - \vol_{N_\infty}(B(p_0,r))}{r^2 \vol_{\E^3}(B(0,r))}
=\lim_{r\to 0} \frac{(4/3)\pi r^3 -  4\pi (2 r)}{(4/3)\pi r^5 }
=-\infty
\ee
as claimed.
\end{proof}

\section{Method II: Sewing a Sequence of Manifolds}\label{S:method1-seqmflds}

In order to prove the examples  Section~\ref{S:sewn-admPMT} and our upcoming paper \cite{BS-tori}
we need to develop 
a more general technique than Method I.  Here we start with a 
converging sequence of 
Riemannian manifolds with a Riemannian limit
and sew regions in that sequence to create
a new sequence of 
Riemannian manifolds with a pulled limit.
More precisely, we consider a sequence of smooth Riemanninan manifolds $\{M_{j}^{3}\}_{j \in \N}$ converging in the biLipschitz sense to $M_{\infty}^{3}$ (also a smooth manifold) and, for each manifold $M_{j}^{3}$, construct its sequence of increasingly tightly sewn manifolds $\{N_{j,i}^{3}\}_{i \in \N}$ scrunching a region $A_{j,0}$. Then the following Theorem states, under suitable hypotheses, that a sequence of sewn manifolds $\{N_{j}^{3}\}$ created from $\{N_{j,i}^{3}\}$ converges in the Gromov-Hausdorff and Intrinsic-Flat sense to $N_{\infty}$ which is a pulled space created by scrunching the region $A_{\infty,0} \subset M_{\infty}^{3}$ to a point.

\begin{thm}\label{prop-seq-sewn}
Given a sequence of compact  $M_j^3$
each with a compact region $A_{j,0}\subset M_j^{3}$ with a tubular neighborhood, $A_j$, of constant sectional curvature %, $K_{j}$, 
satisfying the hypotheses of Proposition~\ref{sewn-region}.
We assume $M_j^3$  converge in the biLipschitz sense 
to $M_\infty^3$  and the regions $A_{j,0}$ 
converge to compact set 
$A_{\infty,0} \subset M_\infty^3$ in the sense that
there exists biLipschitz maps
\be
\psi_j:M_j^3\to M_\infty^3
\ee
such that 
\be\label{lippsijto1}
L_j=\log\Lip(\psi_j)+\log \Lip(\psi_j^{-1}) \to 0
\ee
and
\be\label{psiA}
\psi_j(A_{j,0})=A_{\infty,0}.
\ee
Then there exists $\delta_j \to 0$ and,
applying Proposition~\ref{sewn-region}
to $M^3=M_j^3$ to sew the regions $A_{0}=A_{j,0}$ with $\delta=\delta_j$, to obtain sewn manifolds $N^3=N_j^3$,
we obtain a sequence $N_j^3$ such that
\be\label{diagGHto}
N_j^{3} \GHto N_\infty
\ee
and
\be
N_j^{3} \Fto N_{\infty,0}
\ee
where $\overline{N}_{\infty,0} = N_\infty$ and $N_\infty$ is the metric space created by taking $M_\infty^3$ and pulling the region $A_{\infty,0}$ to a point as in 
Lemma~\ref{pulled-subset-1}--Lemma~\ref{pulled-subset-2}.  

If, in addition, the regions $A_{j,0}$ satisfy $\mathcal{H}^{3}(A_{j,0})=0$, then 
the sequence $N_j^3$ also converges in the the metric measure sense
\be
N_j^{3} \mGHto N_\infty.
\ee
\end{thm}

\begin{proof}
For each $M_j^3$ in the sequence we can apply Proposition~\ref{sewn-region} to create its increasingly tightly sewn
sequence $N_{j,i}^3$ with $\lim_{i\to \infty}\delta_{j,i}=0$. 
By %the fundamental theorem of sewn manifolds, 
Theorem~\ref{thm-seq-sewn}, we know that
\be
N_{j,i} \GHto N_{j,\infty} \textrm{ as } i\to \infty
\ee
and
\be\label{jsubiIFconv}
N_{j,i} \Fto N_{j,0} \textrm{ as } i\to \infty
\ee
where $N_{j,0}$ is an integral current space satisfying $\bar{N}_{j,0} = N_{j,\infty}$ and $N_{j,\infty}$ is the metric space created by taking $M_j^3$ and pulling the region $A_{j,0}$ to a point as in Lemma~\ref{pulled-subset-1}. 
For each $j$ take $i_j$ sufficiently large that
\be
\delta_{i,j}<1/j \qquad \forall i \ge i_j 
\ee
\be
d_{GH}(N_{j,i}, N_{j,\infty})< 1/j \qquad \forall i \ge i_j 
\ee
and
\be
d_{\mathcal{F}}(N_{j,i}, N_{j,0})< 1/j \qquad \forall i \ge i_j 
\ee
and, in the cases we have metric measure convergence, 
\be
|\mass(N_{j,i})-\mass(N_{j,0})|<1/j.
\ee

We choose $\delta_j=\delta_{j,i_j}$ and take
$N_j^{3}=N_{j,i_j}^{3}$.   

By the triangle inequality 
we need only prove:
\be
d_{GH}(N_{j,\infty}, N_\infty)\to 0
\ee
and
\be
d_{\mathcal{F}}(N_{j,0}, N_{\infty,0}) \to 0
\ee
and, in the cases we have measure convergence, 
\be
\mass(N_{j,0}) \to \mass(N_\infty).
\ee
where $N_\infty$ is the metric space created by taking $M_\infty^3$ and pulling the region $A_{\infty,0}$ to a point as in Lemma~\ref{pulled-subset-1}.

Observe that %by Lemma~\ref{psi-1Lip} 
there are 1-Lipschitz maps 
\be
P_{j}: M_{j} \to N_{j,\infty} \quad \text{and} \quad P_{\infty}: M_{\infty} \to N_{\infty}
\ee
defined by
\be
P_{j}(x) =
\begin{cases}
p_{j,0} \qquad &\forall\, x\in A_{j,0}\\
x \qquad &\forall\, x\notin A_{j,0}
\end{cases}
\quad \text{and} \quad %
P_{\infty}(x) =
\begin{cases}
p_{\infty,0} \qquad &\forall\, x\in A_{\infty,0}\\
x \qquad &\forall\, x\notin A_{\infty,0}.
\end{cases}
\ee

Next, define invertible maps
\be\label{defPjandPinfty}
\bar{\psi}_{j}: N_{j,\infty} \to N_{\infty} 
\quad \text{and} \quad %
\bar{\psi}_{j}\inv: N_{\infty} \to N_{j,\infty} 
\ee
by 
\be\label{defpsibar}
\bar{\psi}_{j}(x) = P_{\infty}(\psi_{j}(x))
\quad \text{and} \quad %
\bar{\psi}_{j}\inv(y) = P_{j}(\psi_{j}\inv(y))
\ee
Note that (\ref{psiA}) implies these are well-defined. 

Let $x_{1}, x_{2} \in N_{j,\infty}$, then recall that by definition of the distance $d_{N_{j,\infty}}$ in (\ref{pulled-set-def1})-(\ref{pulled-set-def3}), we have three cases. To simplify notation we shall suppress the $P_{j}$ and $P_{\infty}$ from below. %as we did in FILLIN.

In Case I, assume that
\be
d_{N_{j,\infty}}(x_{1},x_{2})=d_{M_{j}}(x_{1},x_{2})
\ee
so $x_{1},x_{2} \notin A_{j,0}$. By (\ref{defPjandPinfty})-(\ref{defpsibar}) and using that $\psi_{j}$ is Lipschitz we have
\begin{align*}
d_{N_{\infty}}(\bar{\psi}_{j}(x_{1}),\bar{\psi}_{j}(x_{2})) &\le d_{M_{\infty}}(\psi_{j}(x_{1}),\psi_{j}(x_{2}))\\
	&\le \Lip(\psi_{j})d_{M_{j}}(x_{1},x_{2}) \\
	&= \Lip(\psi_{j})d_{N_{j,\infty}}(x_{1},x_{2}) 
\end{align*}
so that $\Lip(\bar{\psi}_{j}) \le \Lip(\psi_{j})$.

In Case II, assume that 
\be 
d_{N_{j,\infty}}(x_{1},x_{2})=d_{M_{j}}(x_{1},A_{j,0})+d_{M_{j}}(x_{2},A_{j,0})
\ee
with $x_{1},x_{2} \notin A_{j,0}$. By (\ref{defPjandPinfty})-(\ref{defpsibar}) and using that $\psi_{j}$ is Lipschitz we have
\begin{align*}
d_{N_{\infty}}(\bar{\psi}_{j}(x_{1}),\bar{\psi}_{j}(x_{2})) &\le d_{M_{\infty}}(\psi_{j}(x_{1}),A_{\infty,0})+ d_{M_{\infty}}(\psi_{j}(x_{2}),A_{\infty,0})
		\qquad \text{($\triangle$ ineq.)}\\
	&\le \Lip(\psi_{j})\left( d_{M_{j}}(x_{1},A_{j,0})+ d_{M_{j}}(x_{2},A_{j,0})\right)\\
	&= \Lip(\psi_{j})d_{N_{j,\infty}}(x_{1},x_{2}) 
\end{align*}
so that $\Lip(\bar{\psi}_{j}) \le \Lip(\psi_{j})$.

Finally, in Case III, assume that $x_{2}=p_{j,0}$ and
\be
d_{N_{j,\infty}}(x_{1},p_{j,0})=d_{M_{j}}(x_{1},A_{j,0})
\ee
with $x_{1} \ne p_{j,0}$, or $x_{1} \notin A_{j,0}$. Then since $\bar{\psi}_{j}(x_{2})=p_{\infty,0}$ and using that $\psi_{j}$ is Lipschitz we have
\begin{align*}
d_{N_{\infty}}(\bar{\psi}_{j}(x_{1}),\bar{\psi}_{j}(x_{2})) &= d_{N_{\infty}}(\bar{\psi}_{j}(x_{1}),p_{\infty,0}) = d_{M_{\infty}}(\psi_{j}(x_{1}),A_{\infty,0})\\
	&\le \Lip(\psi_{j}) d_{M_{j}}(x_{1},A_{j,0})\\
	&= \Lip(\psi_{j})d_{N_{j,\infty}}(x_{1},p_{j,0}) 
\end{align*}
so that $\Lip(\bar{\psi}_{j}) \le \Lip(\psi_{j})$.

Thus, in all three cases we have shown that $\bar{\psi}_{j}$ is Lipschitz with $\Lip(\bar{\psi}_{j}) \le \Lip(\psi_{j})$. 
One can similarly verify that $\bar{\psi}_{j}\inv$ is Lipschitz with $\Lip(\bar{\psi}_{j}\inv) \le \Lip(\psi_{j}\inv)$. 

So,
\be
\bar{L}_{j} = \log \Lip(\bar{\psi}_{j})+ \log \Lip(\bar{\psi}_{j}\inv) \le L_{j} \to 0
\ee
by (\ref{lippsijto1}).

Thus $N_{j,\infty}$ converges to $N_\infty$ in the Lipschitz
sense. 
Since Lipschitz convergence implies Gromov-Hausdorff convergence \cite{Gromov-metric} we have
\be
N_{j,\infty} \GHto N_\infty.
\ee
Moreover, we have Intrinsic Flat convergence,
\be
N_{j,\infty} \Fto N_{\infty,0},
\ee
by Theorem~5.6 of Sormani--Wenger \cite{SW-JDG} where $N_{\infty,0} \subset N_{\infty}$ with $\bar{N}_{\infty,0} = N_{\infty}$. 

To finish, observe that the statement about metric measure convergence now follows from Lemma~\ref{L:mGHconv}. 
\end{proof}

\section{Sequences with ADM mass to 0}\label{S:sewn-admPMT}

We apply the results of the previous Section~\ref{S:method1-seqmflds} together with to the examples of the second author with Lee \cite{LeeSormani1} to construct new examples of sequences of asymptotically flat manifolds with ADM mass decreasing to zero which converge in the (pointed) Intrinsic Flat- and Gromov-Hausdorff-sense to a (weak) limit that is not the standard flat Euclidean space.

%%%%%%
\subsection{Sequences with Stripes of Constant Curvature}\label{S:3DmanfifoldswithPSC}

We review the construction of rotationally symmetric manifolds with nonnegative scalar curvature and small ADM mass following \cite{LeeSormani1}.

\begin{defn}  \label{def-rot-sym}
Let $\RS_3$ be the class of complete $3$-dimensional 
$SO(3)$-rotationally symmetric smooth Riemannian manifolds of nonnegative scalar curvature which have no closed interior minimal hypersurfaces and either have no boundary or have a boundary which is a stable minimal hypersurface.  
\end{defn}

%The following lemmas were proven by the second author in joint work with Lee \cite{LeeSormani1}.

Here one can find simple formulas relating Hawking mass and scalar curvature, and observe that Hawking mass is increasing to the ADM mass. In fact one has an embedding into Euclidean space: 

\begin{lem} (\cite{LeeSormani1}) \label{lem-graph}
Given $(M^3,g) \in \RS_3$, we can find a rotationally
symmetric Riemannian isometric embedding of $M^3$
into Euclidean space
as the graph of some radial function $z=z(r)$ satisfying $z'(r)\geq 0$.  In graphical coordinates, we have 
\be \label{eqn-graph-g}
g=(1+[z'(r)]^2)dr^2 + r^2 g_0,
\ee
with $r\ge r_{min}$ and the following formulae for scalar curvature, area, mean curvature, Hawking mass and its derivative
in terms of the radial coordinate $r$:
\begin{alignat}{1}
\mathrm{R} (r) 
&= \frac{2}{1+(z')^2}\left(\frac{z'}{r}\right)\left(\frac{z'}{r}+ \frac{2z''}{1+(z')^2}\right)\\
A (r)& =\omega_{2} r^{2}\\
H (r)& = \frac{2}{r\sqrt{1+(z')^2}}\\
\mathrm{m}_{\mathrm{H}} (r)& = \frac{r}{2}\left(\frac{(z')^2}{1+(z')^2}\right) \label{mHgraph} \\
\mathrm{m}_{\mathrm{H}}' (r)& = \frac{r^{2}}{4} \mathrm{R} \label{mH'}
\end{alignat}
This Riemannian isometric embedding is unique up to a choice of $z_{min}=z(r_{min})$.
\end{lem}

\begin{lem}(\cite{LeeSormani1})\label{lem-ex}
There is a bijection between elements of
$\RS_3$ and increasing functions
$\mathrm{m}_{\mathrm{H}}:[\rmin,\infty)\to\R$ such that 
\be \label{lem-ex-1}
\mathrm{m}_{\mathrm{H}}(\rmin)=\frac{1}{2} \rmin
\ee 
and
\be
\mathrm{m}_{\mathrm{H}}(r)<\frac{1}{2}r
\ee 
for $r>\rmin\ge 0$.  In this section we will call these functions
\emph{admissible Hawking mass functions}.  

Given an admissible Hawking function, the function $z: [0, \infty) \to \R$ defined via the formula
\be \label{lem-ex-m-H-to-z}
 z(\bar{r}) = \int_{\rmin}^{\bar{r}} \sqrt{ \frac{2\mathrm{m}_{\mathrm{H}}(r)}
{r-2\mathrm{m}_{\mathrm{H}}(r)}}\,dr
\ee
determines a rotationally symmetric manifold in $\RS_3$. 
\end{lem}%end lem

In particular taking a constant Hawking mass, $m_H(r)=m_0$, we have
\be\label{lem-schwarzchild}
z(r) = \sqrt{8m_0\left( r-2m_0\right)}, \quad r \in [2m_0,\infty) 
\ee
with metric
\be
g = \left( 1+ \frac{2m_{0}}{r-2m_{0}}\right) dr^{2} + r^{2} g_{0},
\ee
which is half of the Riemannian Schwarzschild space. Notice these examples are scalar flat since $\mathrm{R}=0$ by (\ref{mH'}).

In order to allow for sewing, we need a region with positive scalar curvature. The second author and Lee show that one can create ``stripes'' of constant curvature $K>0$ within the class $\RS_{3}$ using admissible Hawking functions:

\begin{lem} (\cite{LeeSormani1})\label{lem-stripe}
A manifold $M^3\in \RS_{3}$ has constant sectional curvature, $K>0$,
on $r^{-1}(a,b)$, for $\rmin < a < b$, iff $r^{-1}(a,b)$ is an annulus in a sphere of
radius $1/K^{1/2}$ iff $\mathrm{m}_{\mathrm{H}}(r)=r^3 K/2$ for $r \in (a,b)$.
\end{lem}%end lem

See Figure~\ref{fig-lem-stripe}. 

\begin{figure}[htbp]
\begin{center}
\includegraphics[width=4in]{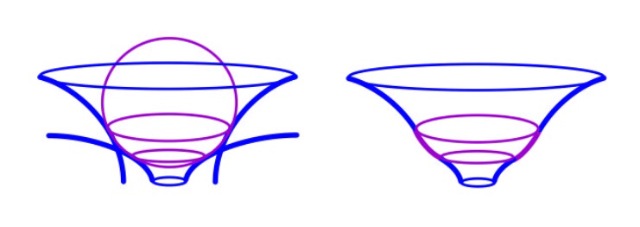}
\caption{An asymptotically flat manifold with a stripe of positive scalar curvature whose ADM mass is small.}
\label{fig-lem-stripe}
\end{center}
\end{figure}

We will need one final lemma from \cite{LeeSormani1}:

\begin{lem}(\cite{LeeSormani1})\label{ex-stripes}
Fix  $\alpha>0$.
Given any increasing sequence, 
\be
\{r_1, r_2,...\}\subset [\mathrm{m}_{\mathrm{fix}}/2,\infty),
\ee
there exists $M^3\in \RS_3$ with constant
sectional curvature on stripes $r^{-1}(a_i,b_i)$
where $(a_i, b_i)\subset [r_{2i-1}, r_{2i}]$
and $\mathrm{m}_{\mathrm{ADM}}(M)<\alpha$ 
and $\partial M=\emptyset$.
\end{lem}

Based on Lemma~\ref{ex-stripes}, the second author and Lee constructed examples of a sequence of asymptotically flat manifolds with $ADM$ mass decreasing to zero with an increasing number of long wells that converged in the pointed Intrinsic Flat sense to euclidean space but have no Lipschitz or Gromov-Haudroff converging subsequences:

\begin{example}\label{LeeSormani-manywells} (\cite{LeeSormani1}) % many wells
There exists a sequence of asymptotically flat manifolds $M_{i}^{3}$ with no interior minimal surfaces and empty boundary and $\lim_{i \to \infty} \mathrm{m}_{\mathrm{ADM}}(M_{i}^{3}) = 0$
such that
for any $\alpha_{0},D>0$ the sequence of regions $T_{D}(\Sigma) \subset M_{i}^{3}$ where $\vol_{2}(\Sigma)=\alpha_{0}$ converge in the intrinsic flat sense to $T_{D}(\Sigma) \subset \E^{3}$ but do not even have Lipschitz or Gromov-Hausdorff converging subsequences. 
\end{example}

\subsection{Failing Almost Rigidity of the Positive Mass Theorem}

In this section we apply the sewing technique to sequences of asymptotically flat manifolds with $ADM$ mass decreasing to zero to give new examples which converge (by Theorem~\ref{prop-seq-sewn}) in the pointed Gromov-Hausdorff and pointed Intrinsic Flat sense to weak limits that are not the standard flat euclidean space.

To complete our constructions, we need a region of constant positive sectional curvature to sew, we use Lemma~\ref{lem-stripe}. See Figure~\ref{fig-lem-stripe} .   

Unlike the sequence of Lee and Sormani, the sequence in Example~\ref{sewnAFnotEuclid-bubble} converges in pointed Gromov-Hausdorff sense. %and pointed Intrinsic Flat sense to a limit that is not isometric to flat euclidean space.  

\begin{figure}[htbp]
\begin{center}
\includegraphics[width=4in]{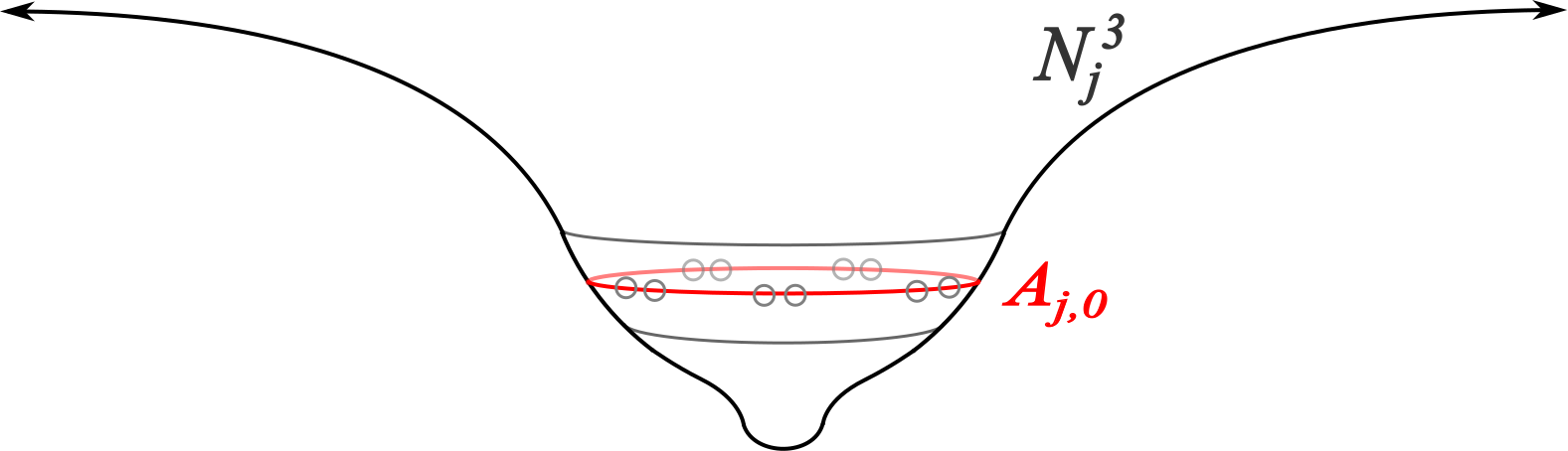}
\caption{An element of the sequence of asymptotically flat sewn manifolds constructed in Example~\ref{sewnAFnotEuclid-bubble}.}
\label{fig-sewnschwarzchild-sphere}
\end{center}
\end{figure}

% pulling a curve
\begin{example}\label{sewnAFnotEuclid-bubble} % NEW
There exists a sequence of asymptotically flat manifolds $N_{j}^{3}$ with 
nonnegative scalar curvature, empty boundary and $\lim_{j \to \infty} \mathrm{m}_{\mathrm{ADM}}(M_{j}^{3}) = 0$ 
that
converges in the pointed Gromov-Hausdorff and pointed Intrinsic Flat sense to $\E^{3}$ with a closed curve %$2$-sphere
pulled to a point $p_{0}$ in the following sense: 
given any $\alpha_{0}>0$ and $D>0$, 
%the sequence of tubular neighborhoods $T_{D}(\Sigma_{\alpha_{0}}) \subset N_{j}^{3}$
\be
d_{GH}(T_{D}(\Sigma_{\alpha_{0}}) \subset N_{j}^{3}, 
	T_{D}(\Sigma_{\alpha_{0}}) \subset N_{\infty}) \to 0
\ee
and
\be
d_{\Fm}(T_{D}(\Sigma_{\alpha_{0}}) \subset N_{j}^{3},  
	T_{D}(\Sigma_{\alpha_{0}}) \subset N_{\infty}) \to 0
\ee
where $N_{\infty}$ which is $\E^{3}$ with a closed curve %$2$-sphere 
pulled to a point $p_{0}$ as in Lemma~\ref{pulled-subset-1}
and
$\Sigma_{\alpha_{0}}$ is the surface with $\vol_{2}(\Sigma_{\alpha_{0}})=\alpha_{0}$.  
Thus, $N_{\infty}$ is homeomorphic to $\E^{3} \disjointunion_{p_{0}} \mathbb{S}^{3}$, which is a wrinkled $\E^{3}$ with a wrinkled sphere attached. See Figure~\ref{fig-sewnschwarzchild-sphere}.

\end{example}

\begin{proof}
Let $\alpha_{0}, D>0$ be given. 
Set
\be\label{anchor-radii}
	r_{0} = \left(\frac{\alpha_{0}}{4\pi}\right)^{1/2}
	\qquad \text{and} \qquad
%	r_{D-} = \inf\{ r(p) \mid p \in T_{D}(\Sigma_{0}) \} \\ dont choose like this since you need the manifolds M first...point is to choose them universally 
	r_{1} = 
		\begin{cases}
		r_{0}-D/2, \quad &r_{0}-D \ge 0\\
		r_{0}/2, &\text{otherwise}
		\end{cases}
\ee

Fix $j \in \N$. Let $\delta_j=1/j$,
\be\label{stripe-radii}
	r_{j,2} = r_{1}-1/j \qquad \text{and } \qquad
	r_{j,3} = r_{1}+1/j.
\ee

By Lemma~\ref{ex-stripes}, there exists rotationally symmetric manifolds $\bar{M}_{j}^{3} \in \RS_{3}$ with $\textrm{m}_{\textrm{ADM}}(\bar{M}_j^3)<\delta_j$ and with constant sectional curvature $K_j>0$ on the stripe 
\be
	r^{-1}(a_{j},b_{j}) \subset r^{-1}[r_{j,2},r_{j,3}],
\ee
for some $a_{j}<b_{j}$ in the interval $[r_{j,2},r_{j,3}]$.

Let $M_{j}^{3}=Cl(T_{D}(\Sigma_{j,0})) \subset \bar{M}_{j}^{3}$ and $M_{\infty}^{3} = Cl(T_{D}(\Sigma_{0})) \subset \E^{3}$, where $\Sigma_{j,0}$ and $\Sigma_{0}$ are the surfaces with area equal to $\alpha_{0}$ and, thus, are at a distance of $r_{0}$ from the axis (in graphical coordinates). Further, let $A_{j,0}$ be closed geodesic circle inside $r^{-1}((a_{j}+b_{j})/2)$ in $M_{j}^{3}$. Let $A_{\infty,0}$ be the circle centered at $0$ in $\E^{3}$ of radius $r_{1}$. Observe that by our choices in (\ref{anchor-radii}) and (\ref{stripe-radii}) the stripe of constant sectional curvature $r^{-1}(a_{j},b_{j})$ always belongs to $M_{j}$ (including $j=\infty)$.

The tubular neighborhoods $M_{j}^{3}$ and $M_{\infty}^{3}$ are isometrically embedded into $\E^{4}$ by Lemma~\ref{lem-graph}, so there exist biLipschitz maps $\psi_{j}: M_{j}^{3} \to M_{\infty}^{3}$ with $\psi_{j}(A_{j,0})=A_{\infty,0}$ and (\ref{lippsijto1}).
Moreover,  
$M_{j}^{3}$ and $M_{\infty}^{3}$ are compact, so we can apply Theorem~\ref{prop-seq-sewn} to obtain a sequence of sewn manifolds $N_{j}^{3}$ satisfying
\be
d_{mGH}(N_{j}^{3}, N_{\infty}) \to 0
\ee
and
\be
d_{\Fm}(N_{j}^{3}, N_{\infty}) \to 0
\ee
where $N_{\infty}$ is $M_{\infty}$ with the circle $A_{\infty,0}$ pulled to a point $p_{0}$ as in Proposition~\ref{pulled-string}.
\end{proof}

Similarly, we can pull a 2-sphere to a point instead of a curve:

% pulling a 2-sphere
\begin{example} % NEW
There exists a sequence of asymptotically flat manifolds $N_{j}^{3}$ with 
nonnegative scalar curvature, empty boundary and $\lim_{j \to \infty} \mathrm{m}_{\mathrm{ADM}}(M_{j}^{3}) = 0$ 
that
converges in the pointed Gromov-Hausdorff and pointed Intrinsic Flat sense to $\E^{3}$ with a closed $2$-sphere
pulled to a point $p_{0}$ in the following sense: 
given any $\alpha_{0}>0$ and $D>0$, 
%the sequence of tubular neighborhoods $T_{D}(\Sigma_{\alpha_{0}}) \subset N_{j}^{3}$
\be
d_{GH}(T_{D}(\Sigma_{\alpha_{0}}) \subset N_{j}^{3}, 
	T_{D}(\Sigma_{\alpha_{0}}) \subset N_{\infty}) \to 0
\ee
and
\be
d_{\Fm}(T_{D}(\Sigma_{\alpha_{0}}) \subset N_{j}^{3},  
	T_{D}(\Sigma_{\alpha_{0}}) \subset N_{\infty}) \to 0
\ee
%\be
%N_j^3 \mGHto N_\infty \textrm{ and } N_j^3 \Fto N_{\infty} 
%\ee
where $N_{\infty}$ which is $\E^{3}$ with a closed $2$-sphere 
pulled to a point $p_{0}$ as in Lemma~\ref{pulled-subset-1}
and
$\Sigma_{\alpha_{0}}$ is the surface with $\vol_{2}(\Sigma_{\alpha_{0}})=\alpha_{0}$.  
Thus, $N_{\infty}$ is homeomorphic to $\E^{3} \disjointunion_{p_{0}} \mathbb{S}^{3}$, which is a wrinkled $\E^{3}$ with a wrinkled sphere attached. See Figure~\ref{fig-sewnschwarzchild-sphere}.

\end{example}

The proof is nearly identical to that of Example~\ref{sewnAFnotEuclid-bubble}.

In the next example, we obtain a limit of asymptotically flat manifolds with ADM decreasing to zero that is homeomorphic to $\E^{3}$ but not isometric. 

\begin{figure}[htbp]
\begin{center}
\includegraphics[width=4in]{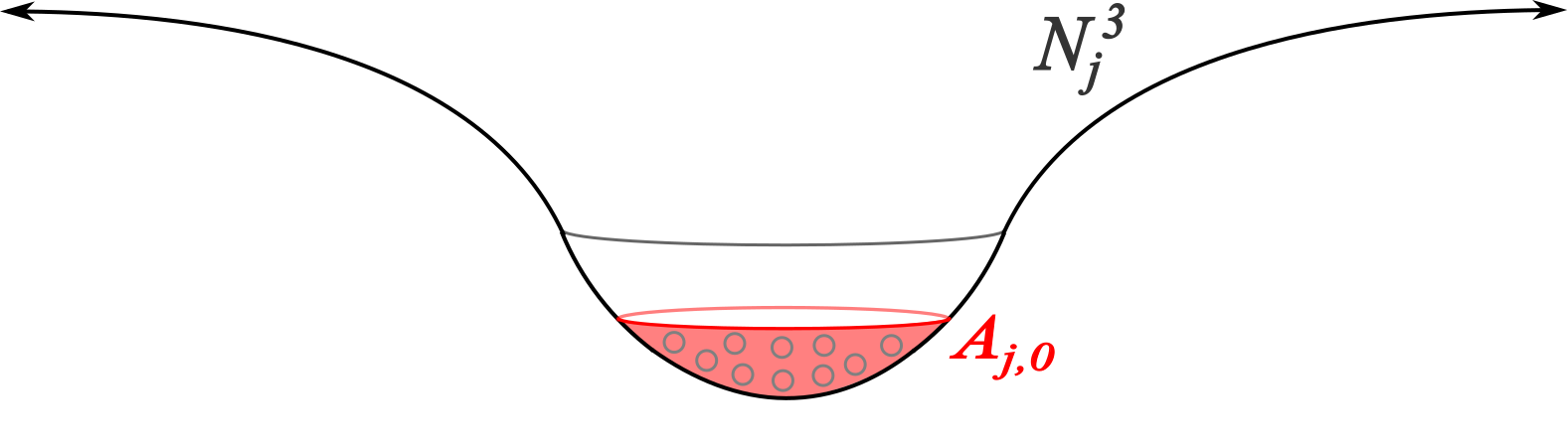}
\caption{An element of the sequence of asymptotically flat sewn manifolds constructed in Example~\ref{sewnAFnotEuclid-homeo}.}
\label{fig-sewnschwarzchild-ball}
\end{center}
\end{figure}

% pulling a 3-ball
\begin{example}\label{sewnAFnotEuclid-homeo} % NEW
There exists a sequence of asymptotically flat manifolds $N_{j}^{3}$ with 
nonnegative scalar curvature, empty boundary and $\lim_{j \to \infty} \mathrm{m}_{\mathrm{ADM}}(M_{j}^{3}) = 0$ 
that
converges in the pointed Gromov-Hausdorff and pointed Intrinsic Flat sense to $\E^{3}$ with a closed unit $3$-ball
pulled to a point $p_{0}$ in the following sense: 
given any $\alpha_{0}>0$ and $D>0$ suffificently large, 
%the sequence of tubular neighborhoods $T_{D}(\Sigma_{\alpha_{0}}) \subset N_{j}^{3}$
\be
d_{GH}(T_{D}(\Sigma_{\alpha_{0}}) \subset N_{j}^{3}, 
	T_{D}(\Sigma_{\alpha_{0}}) \subset N_{\infty}) \to 0
\ee
and
\be
d_{\Fm}(T_{D}(\Sigma_{\alpha_{0}}) \subset N_{j}^{3},  
	T_{D}(\Sigma_{\alpha_{0}}) \subset N_{\infty}) \to 0
\ee
where $N_{\infty}$ which is $\E^{3}$ with a closed $3$-ball 
pulled to a point $p_{0}$ as in Lemma~\ref{pulled-subset-1}
and
$\Sigma_{\alpha_{0}}$ is the surface with $\vol_{2}(\Sigma_{\alpha_{0}})=\alpha_{0}$.  
\end{example}

\begin{proof}
Let $\alpha_{0}>0$ be given. 
Set
\be\label{anchor-radii-2}
	r_{0} = \left(\frac{\alpha_{0}}{4\pi}\right)^{1/2}
\ee
and let $D$ be large enough that $D>r_{0}$. 

Fix $j \in \N$. Let $\delta_j=1/j$ and 
\be\label{stripe-radii-2}
	r_{j,2} = 0 \qquad \text{and } \qquad
	r_{j,3} = r_{0}/2.
\ee

By Lemma~\ref{ex-stripes}, there exists rotationally symmetric manifolds $\bar{M}_{j}^{3} \in \RS_{3}$ with $\textrm{m}_{\textrm{ADM}}(\bar{M}_j^3)<\delta_j$ and with constant sectional curvature $K_j>0$ on the stripe 
\be
	r^{-1}(0,b_{j}) \subset r^{-1}[r_{j,2},r_{j,3}],
\ee
for some $a_{j}=0<b_{j}$ in the interval $[r_{j,2},r_{j,3}]$.

Let $M_{j}^{3}=Cl(T_{D}(\Sigma_{j,0})) \subset \bar{M}_{j}^{3}$ and $M_{\infty}^{3} = Cl(T_{D}(\Sigma_{0})) \subset \E^{3}$, where $\Sigma_{j,0}$ and $\Sigma_{0}$ are the surfaces with area equal to $\alpha_{0}$ and, thus, are at a distance of $r_{0}$ from the axis (in graphical coordinates). 
Further, let $A_{j,0}$ be closed 3-ball $r^{-1}(0,b_{j}/2)$ in $M_{j}^{3}$. 
Let $A_{\infty,0}$ be 3-ball centered at $0$ in $\E^{3}$ of radius $r_{0}/2$. Observe that by our choices in (\ref{anchor-radii-2}) and (\ref{stripe-radii-2}) the stripe of constant sectional curvature $r^{-1}(0,b_{j})$ always belongs to $M_{j}$ (including $j=\infty)$ and it is at the bottom.

The tubular neighborhoods $M_{j}^{3}$ and $M_{\infty}^{3}$ are isometrically embedded into $\E^{4}$ by Lemma~\ref{lem-graph}, so there exist biLipschitz maps $\psi_{j}: M_{j}^{3} \to M_{\infty}^{3}$ with $\psi_{j}(A_{j,0})=A_{\infty,0}$ and (\ref{lippsijto1}).
Moreover,  
$M_{j}^{3}$ and $M_{\infty}^{3}$ are compact, so we can apply Theorem~\ref{prop-seq-sewn} to obtain a sequence of sewn manifolds $N_{j}^{3}$ satisfying
\be
d_{GH}(N_{j}^{3}, N_{\infty}) \to 0
\ee
and
\be
d_{\Fm}(N_{j}^{3}, N_{\infty}) \to 0
\ee
where $N_{\infty}$ is $M_{\infty}$ with the ball $A_{\infty,0}$ pulled to a point $p_{0}$ as in Proposition~\ref{pulled-string}.
\end{proof}

\bibliographystyle{alpha}
\bibliography{basilio2}
\end{document}